\documentclass{ieeeconf}
\IEEEoverridecommandlockouts

\usepackage{cite}
\usepackage{algorithmic}
\usepackage{graphicx}
\usepackage{textcomp}
\usepackage{xcolor}
\usepackage{amsmath,amssymb,amsthm,amsfonts}
\usepackage{fancyhdr}
\newtheorem{theorem}{Theorem}
\newtheorem{corollary}{Corollary}

\newtheorem{lemma}{Lemma}
\newtheorem{remark}{Remark}

\begin{document}
%\begin{frontmatter}
\pagenumbering{arabic}
%https://www.overleaf.com/project/658c6719d3c39af42a000cae
%\fancyfoot[L]{This work has been submitted to the 2024 Conference on Decision and Control}
\title{An SVD-like Decomposition of Bounded-Input Bounded-Output Functions} 

% Title, preferably not more than 10 words.
\author{Brian Charles Brown$^1$, Michael King$^1$, Sean Warnick$^{1}$, Enoch Yeung$^{2,3}$, David Grimsman$^{1}$
\thanks{1 - Department of Computer Science, Brigham Young University, UT, 84602;
2 - Department of Mechanical Engineering, University of California, Santa Barbara, CA 93106;
3- Department of Bioengineering, University of California, Santa Barbara, CA 93106.
This work was funded by DOE Grant \#SC0021693, and has been submitted to the 2024 Conference on Decision and Control.
Correspondence should be addressed to Brian Brown at \texttt{bcbrown365@gmail.com}.}}
%\author[First]{Brian Charles Brown} 
%\author[Second]{Michael King} 
%\author[Third]{Sean Warnick}
%\author[Fourth]{Enoch Yeung}
%\author[Fifth]{David Grimsman}

\maketitle

%\address[First]{Department of Computer Science, 
%   Brigham Young University, UT 84602 USA (e-mail: bcbrown365@gmail.com).}
%\address[Second]{Department of Computer Science, 
%   Brigham Young University, UT 84602 USA (e-mail: mhking71@gmail.com).}
%\address[Third]{Department of Computer Science, 
%   Brigham Young University, UT 84602 USA (e-mail: something@gmail.com).}
%\address[Fourth]{Department of Biological Engineering, 
%   University of California, Santa Barbara, CA 84602 USA (e-mail: something@.com).}  
%\address[Fifth]{Department of Mathematics, 
%   Brigham Young University, UT 84602 USA (e-mail: something@gmail.com).}
   
\begin{abstract}                % Abstract of 50--100 words
The Singular Value Decomposition (SVD) of linear functions facilitates the calculation of their 2-induced norm and row and null spaces, hallmarks of linear control theory.
In this work, we present a function representation that, similar to SVD, provides an upper bound on the 2-induced norm of bounded-input bounded-output functions, as well as facilitates the computation of generalizations of the notions of row and null spaces.
Borrowing from the notion of ``lifting" in Koopman operator theory, we construct a finite-dimensional lifting of inputs that relaxes the unitary property of the right-most matrix in traditional SVD, $V^*$, to be an injective, norm-preserving mapping to a slightly higher-dimensional space.
\end{abstract}

%\begin{keyword}
%Riesz representation, singular value decomposition, L-infinity, neural network explainability, nonlinear synthesis
%\end{keyword}

%===============================================================================

\section{Introduction}

Decomposing a function $f: \mathbb{R}^n \to \mathbb{R}^p$ into smaller or more manageable terms can often lead to valuable insights and tools to analyze $f$ \cite{zupan1997machine}. For instance, $f$ can be represented as the weighted sum of sinusoids in the Fourier series, allowing one to identify which frequencies are most significant in computing the output of $f$ \cite{fourier21memoire}. Other decomposition methods include using radial basis functions \cite{Buhmann_2000}, wavelets \cite{qian2007wavelet}, or polynomials \cite{barton1985polynomial}, each giving a different perspective on $f$.  

In the case where $f$ is a linear function of finite-dimensional inputs and outputs, then one can use the well-known singular value decomposition (SVD) to identify inputs whose output has a maximum (or minimum) increase in magnitude \cite{eckart1936approximation}. The $LU$ decomposition and the $QR$ decomposition give insight into $f^{-1}$, in other words, how to solve the set of linear equations defined by $f$ and some output \cite{horn2012matrix}. When $f$ represents a linear dynamical system $\dot{x} = f(x)$ ($f$ maps $\mathbb{R}^n$ to itself) then methods such as Jordan decomposition \cite{brechenmacher2006histoire} or Schur decomposition \cite{horn2012matrix} can be used to identify stability and other properties of the system.

When $f$ is a linear functional, i.e. is a linear map from $\mathbb{R}^n$ to $\mathbb{R}$, the Riesz Representation Theorem implies that $f$ can be associated with a unique vector $u \in \mathbb{R}^n$ such that $f(x) = \langle u, x \rangle$ for any $x \in \mathbb{R}^n$ \cite{bachman2000functional}. This seminal result shows that each linear functional can be represented by an element of its domain, which has widespread benefits in computational physics, for instance. Stacking these functionals leads to a similar result, or decomposition, for linear functions mapping $\mathbb{R}^n$ to $\mathbb{R}^p$.

Another type of decomposition from which this work draws inspiration is that of Koopman. For a nonlinear dynamical system, the states can be ``lifted" to a higher (potentially infinite) dimension, whereby the dynamics of the system are precisely described by the linear Koopman Operator \cite{brunton2022}. Much recent work has been devoted to advancing Koopman Operator Theory (see, for instance \cite{yeung2017learning,mauroy2020koopman,mezic2013analysis,williams2015data}), including a recent trend to impose that liftings either be invertible \cite{MENG2024112795,JIN2024106177} or state-inclusive \cite{johnson2018class}. One key advantage of using the Koopman Operator is that it allows one to use well-known and well-understood tools for analyzing linear dynamical systems, such as the ones described above, in the context of nonlinear systems. While the focus on dynamical systems has generally limited the analysis to functions that map $\mathbb{R}^n$ to itself, our goal in this work is to study more general functions that map $\mathbb{R}^n$ to $\mathbb{R}^p$. 

The main contribution of this paper is to present a novel decomposition for any arbitrary bounded-input bounded-output function $f: \mathbb{R}^n \to \mathbb{R}^p$, such that $\|f(x)\|_2 < c \|x\|_2$ for all $x \in \mathbb{R}^n$ and for some $c \in \mathbb{R}^+ < \infty$. The function $f$ is decomposed into two parts: a linear part and a norm-preserving injective nonlinear part, as stated precisely in  Theorem~\ref{theorem:sufficient}. \emph{The primary benefit of this decomposition is that tools used for analyzing linear functions, such as SVD, can be adapted to analyze $f$.} Indeed, Theorem~\ref{theorem:sufficient} shows that our decomposition is a generalization of the SVD to a large class of nonlinear functions.

We note that other work in the literature has the goal to generalize the SVD. For instance, both \cite{vanloan1976generalizing} and \cite{greenacre1984theory} develop such ideas, but still restricted to linear functions. The work in \cite{dada2020generalized} is aimed at using a generalized SVD for nonlinear dynamical systems, but only to build observers; the generalization is that one decomposes two matrices instead of one. The works in \cite{tao2023nonlinear} and \cite{vaidya2007non} address the scenario where only a finite number of observations are known about $f$, with the goal to identify $f$ by augmenting the data matrix with columns that are functions (or observables) of the original data. While our approach is somewhat similar, the goal of this work is different in that we seek a representation of the function itself. Furthermore, we believe we are unique in enforcing that our observables are norm-preserving, a key to ensuring that the linear part of the decomposition is as descriptive as possible (see Remark~\ref{rem:norm-preserving} after Theorem~\ref{theorem:sufficient}).

%\section{Representation of Scalar BIBO Functionals with Norm-Preserving Transformations}
%Norm-Preserving Liftings: N+1 is all you need}
%\label{sec:norm_preserving}

\subsection{Notation}
\label{sec:notation}

Per notation common in Koopman operator theory, for some linear mapping $K$ and a potentially nonlinear mapping, $g$, we use $(K \circ g) (x)$ to represent the composition of $K$ with $g$.
However, if $K$ and $g$ are finite-dimensional, for example $K \in \mathbb{R}^{p \times m}$ and $g(x) \in \mathbb{R}^{p}$, then we define $(K \circ g)(x) = K g(x)$, i.e. traditional matrix-vector multiplication.
We use $V^*$ to denote the Hermitian transpose of some matrix $V$.
The matrices $U$, $\Sigma$, and $V^*$ will always represent the matrices of the singular value decomposition of some matrix $K$, such that $K = U \Sigma V^*$, with $U \in \mathbb{R}^{p \times p}$, a unitary (and therefore norm-preserving, injective, and surjective) matrix, $V^* \in \mathbb{R}^{m \times m}$, another unitary matrix, and $\Sigma \in \mathbb{R}^{p \times m}$, a real, non-negative, rectangular-diagonal matrix such that the $i$-th diagonal element is given by $\sigma_i$, such that $\sigma_i \ge \sigma_j \ge 0$ for all $i < j$.
Calligraphic letters, e.g. $\mathcal{X}$, are always sets, and $|\mathcal{X}| \in \mathbb{Z}$ is the cardinality of $\mathcal{X}$.

For some mapping $f: \mathbb{R}^n \to \mathbb{R}^p$, $f_i: \mathbb{R}^{n} \to \mathbb{R}$ is the $i$-th component functional of $f$.
For brevity, we will denote $f(x)$ as $f$ and $f_i(x)$ as $f_i$ when $x$ is arbitrary.
We denote the 2-induced norm of $f$ as $\|f\|_{2-2} = \sup_x \frac{\|f(x)\|_2}{\|x\|_2}$, which for a linear $f$, is given by the maximum singular value, $\sigma_1$, of the matrix representation of $f$.
For some mapping $v: \mathbb{R}^n \to \mathbb{R}^m$, then for some particular $x$, $\|v(x)\|_2$ represents the 2-norm of the vector $v(x) \in \mathbb{R}^{m}$.

We use $\mathbf{1}_p$ to represent a vector of all ones of dimension $p$.
Similarly, $e_i$ is a vector of all zeros except a 1 at the $i$-th index.
Occasionally, we will need to refer to the $i$-th element of a particular vector, $x_j$.
In this case, we will use $x_{j,i}$ to denote the $i$-th element of the $j$-th vector $x$.

We will also use element-wise operations on vectors.
For some $\sigma \in \mathbb{R}^{p}$, we define $\sigma^{-1}$, and $\sigma \odot \sigma$, and $\sigma^2$ as follows:

\begin{equation*}
\sigma^{-1} = 
\begin{bmatrix}
\frac{1}{\sigma_1} \\
\frac{1}{\sigma_2} \\
\vdots \\
\frac{1}{\sigma_p} \\
\end{bmatrix},
\quad
\sigma \odot \sigma = \sigma^2 = 
\begin{bmatrix}
\sigma_1^2 \\
\sigma_1^2 \\
\vdots \\
\sigma_1^2 \\
\end{bmatrix}
\end{equation*}

\section{Representation of Vector-Valued  BIBO Functions with Norm-Preserving Transformation}

\begin{theorem}
\label{theorem:sufficient}
Let $f: \mathbb{R}^n \rightarrow \mathbb{R}^p$ be an arbitrary bounded-input bounded-output function.
Then there exists a unitary matrix, $U \in \mathbb{R}^{p \times p}$, a real, non-negative, rectangular-diagonal matrix $\Sigma \in \mathbb{R}^{p \times m}$, and a norm-preserving, injective mapping, \mbox{$v: \mathbb{R}^n \rightarrow \mathbb{R}^m$}, with $m \ge p+n$, such that:

\begin{equation}
\label{eq:functionRepresentation}
f(x) = U \Sigma v (x), \quad \text{for all } x \in \mathbb{R}^n
\end{equation}
\end{theorem}
 \begin{proof}
Let $m=n+p$, $\delta: \mathbb{R}^{n} \to \mathbb{R}^{p}$,
$x_{\delta} := 
\begin{bmatrix} 
\delta(x) \\ 
x 
\end{bmatrix}$,
and \mbox{$v: \mathbb{R}^n \to \mathbb{R}^{m}$} given by: 
\begin{equation}
v(x) := \frac{\|x\|_2}{\|x_{\delta}\|_2}x_{\delta}.
\end{equation}

Notice that for any well-defined function $\delta$, $v$ is both norm-preserving, i.e. $\|v(x)\|_2 = \|x\|_2,\; \forall x \in \mathbb{R}^n$, and injective, i.e. for all $x_1 \neq x_2 \in \mathbb{R}^n, v(x_1) \neq v(x_2)$.
Choosing the appropriate function $\delta$ and a corresponding real, non-negative, rectangular-diagonal matrix:
\begin{align}
\label{eq:sigma}
\Sigma  &:= 
\left[\begin{array}{cccc|c} 
\sigma_1 & 0 & 0 & \dots & \bf{0}\\
0 & \sigma_2 & 0 & \dots & \bf{0}\\
0 & 0 & \ddots & & \bf{0} \\
\vdots & \vdots & &  \sigma_p & \bf{0} \\
\end{array}\right] \in \mathbb{R}^{p \times m}
\end{align}
with (admissible) $\sigma_i \ge \sigma_j \geq 0 $ for $j > i$, to satisfy Equation (\ref{eq:functionRepresentation}), is the key to the proof.  

In order to ensure that $\sigma_i \ge \sigma_j \ge 0$ for all $i > j$, let $f_{q(i)}$ denote the component functional of $f$ with the $i$-th placement in the relative ranking of the 2-induced norms of the component functionals of $f$.
For example, if $\|f_j\|_{2-2}$ were the smallest among all $\|f_i\|_{2-2}, \; i = 1, 2, \dots, p$, then $f_{q(p)} = f_j$, whereas if $\|f_l\|_{2-2}$ were the largest, then $f_{q(1)} = f_l$.
Let $f_q$ represent $f$ with its component functionals re-ordered from largest to smallest induced norm, such that the first index of $f_q = f_{q(1)}$, etc.
Let $\Sigma$ be defined given some values $\sigma_i$, for $i=1,2,\dots, p$, such that:
\begin{align}
\label{eq:sigma}
\overset{p}{\underset{i=1}{\sum}} \frac{\|f_{q(i)}\|_{2-2}^2}{\sigma_i^2} < 1,
\end{align}
and consider $d_i$ given by:
\begin{align}
d_i := \frac{\sigma_i^2 \|x\|_2^2}{f_{q(i)}(x)^2} - 1,
\end{align}
and a matrix, $A \in \mathbb{R}^{p \times p}$, such that:

\begin{equation}
A := 
\begin{bmatrix}
d_1 & -1 & \dots & -1 \\
-1 & d_2 & \dots & -1 \\
\vdots & & \ddots & \vdots \\
-1 & \dots & & d_p
\end{bmatrix}.
\end{equation}

\noindent We conjecture that for some choice of admissible $\sigma$, \mbox{$A \delta^2 = \|x\|_2^2 \bf{1}$}, and that this equality will imply that \mbox{$f(x) = U \Sigma v(x)$} for an appropriate choice of unitary $U$.
To check this, we expand the $i$-th row of $A \delta^2 = \|x\|_2^2 \bf{1}$ , to find:
\begin{align}
\left(\frac{\sigma_i^2 \|x\|_2^2}{f_{q(i)}^2} - 1 \right)\delta_i^2 - \Sigma_{j \neq i}^p \delta_j^2  &= \|x\|_2^2, \\
\left( 1 - \frac{f_{q(i)}^2}{\sigma_i^2 \|x\|_2^2} \right) \delta_i^2 - \frac{f_{q(i)}^2}{\sigma_i^2 \|x\|_2^2} \Sigma_{j \neq i}^p \delta_j^2 &= \frac{f_{q(i)}^2}{\sigma_i^2}, \notag  \\
\delta_i^2 - \left( \frac{\sigma_i^{-2} f_{q(i)}^2 \delta_i^2}{\|x\|_2^2} + \frac{\sigma_i^{-2} f_{q(i)}^2 \Sigma_{j \neq i}^p \delta_j^2}{\|x\|_2^2}\right) &= \sigma_i^{-2}f_{q(i)}^2. \notag \\
\delta_i^2 - \sigma_i^{-2} f_{q(i)}^2 \frac{\left(\delta_i^2  + \Sigma_{j \neq i}^p \delta_j^2 \right)}{\|x\|_2^2} &= \sigma_i^{-2}f_{q(i)}^2. \notag
\end{align}
Stacking and using element-wise operations yields:
\begin{align}
\delta^2 - f_q^2  \odot \sigma^{-2} \frac{\| \delta \|_2^2}{\|x\|_2^2} &=  f_q^2 \odot \sigma^{-2} \implies \notag \\
\delta^2 &= f_q^2 \odot \sigma^{-2} \frac{\|x\|_2^2}{\|x\|_2^2} + f_q^2  \odot \sigma^{-2} \frac{\| \delta \|_2^2}{\|x\|_2^2}, \notag \\
&= f_q^2 \odot \sigma^{-2} \frac{(\Sigma_i^n x_i^2) + (\Sigma_i^p \delta_i^2)}{\|x\|_2^2}, \notag \\
&= f_q^2 \odot \sigma^{-2} \frac{(\Sigma_i^n x_j^2) + (\Sigma_i^p \delta_i^2)}{\|x\|_2} \implies \notag \\
\delta &= f_q \odot \sigma^{-1} \frac{\|x_{\delta}\|_2}{\|x\|_2}, \notag \\
\sigma \odot \delta &= f_q \frac{\|x_{\delta}\|_2}{\|x\|_2}, \notag \\
\sigma \odot \delta \frac{\|x\|_2}{\|x_{\delta}\|_2} &= f_q(x), \notag \\
\Sigma v(x) &= f(x) \label{eq:almost}.
\end{align}

It now remains to identify the $\delta$ that satisfies $A \delta^2 = \|x\|^2 \bf{1}$, which reduces to finding the inverse of $A$ and ensuring that $\delta^2$ is positive so that $\delta$ is real-valued.
We will first address the inverse of $A$ by using the Woodbury matrix identity.
Define $\tilde{A} := \text{Diag}(A) + I$, $\tilde{B} := \mathbf{1}_p$, a column vector of $p$ ones, $\tilde{C} := -1$, $\tilde{D} := \mathbf{1}_p^T$, a row vector of $p$ ones, and $\gamma := -1 + \frac{1}{d_1+1} + \frac{1}{d_2+1} + \dots + \frac{1}{d_p+1}$.
We thus have:
\begin{align*}
A^{-1} &= (\tilde{A} + \tilde{B}\tilde{C}\tilde{D})^{-1}, \\
&= \tilde{A}^{-1} - \tilde{A}^{-1}\tilde{B}(\tilde{C}^{-1}+\tilde{D}\tilde{A}^{-1}\tilde{B})^{-1}\tilde{D}\tilde{A}^{-1}, \\
&= \tilde{A}^{-1} - \tilde{A}^{-1}\textbf{1}_p(-1 + \mathbf{1}_p^T \tilde{A}^{-1}\mathbf{1})^{-1}\mathbf{1}_p^T \tilde{A}^{-1}, \\
&= \tilde{A}^{-1} - \tilde{A}^{-1}\mathbf{1}_p \frac{1}{\text{tr}(\tilde{A}^{-1}) - 1}\mathbf{1}^T_p\tilde{A}^{-1}, \\
\end{align*}
\begin{align*}
A^{-1} =
\begin{bmatrix}
\frac{1}{(d_1+1)} & 0 & \dots & 0\\
0 & \frac{1}{(d_2+1)} & & \vdots \\
\vdots & & \ddots & \\
0 & \dots & & \frac{1}{d_p + 1}
\end{bmatrix} - \\ 
\begin{bmatrix}
\frac{1}{d_1 +1} \\
\vdots \\
\frac{1}{d_p + 1}
\end{bmatrix}
\begin{bmatrix}
\frac{1}{d_1 + 1} & \dots & \frac{1}{d_p + 1}
\end{bmatrix}
\frac{1}{\gamma}.
\end{align*}
Element-wise, this yields
\begin{align}
\delta_i^2 = \|x\|_2^2 \left(\frac{1}{d_i + 1} + \sum_{j=1}^p \frac{-1}{(d_i + 1)(d_j + 1) \gamma}\right) \implies \\
\delta_i = \textrm{sgn}(f_{q(i)}(x)) \|x\|_2 \left(\frac{1}{d_i + 1} + \sum_{j=1}^p \frac{-1}{(d_i + 1)(d_j + 1) \gamma}\right)^{\frac{1}{2}}.
\end{align}
We now see that in order for $\delta$ to be real-valued, the expression $\left(\frac{1}{d_i + 1} + \sum_{j=1}^p \frac{-1}{(d_i + 1)(d_j + 1) \gamma}\right)$ must be positive.
\begin{align}
\delta_i^2 > 0 \implies \notag \\
\|x\|_2\left(\frac{1}{d_i+1} + \overset{p}{\underset{j=1}{\sum}} \frac{-1}{(d_i + 1)(d_j + 1)\gamma} \right) &> 0 \implies \notag \\
\frac{1}{d_i+1} > \overset{p}{\underset{j=1}{\sum}} \frac{1}{(d_i + 1)(d_j + 1)\gamma} \implies \notag \\
1 > \overset{p}{\underset{j=1}{\sum}} \frac{1}{(d_j + 1)\gamma} \label{eq:negative}
\end{align}
We postulate that $\gamma$ must be negative in order for (\ref{eq:negative}) to be true.
Suppose for contradiction that $\gamma >0$.  This leads to:
\begin{align*}
\gamma > \overset{p}{\underset{j=1}{\sum}} \frac{1}{(d_j + 1)} \implies \\
-1 + \overset{p}{\underset{j=1}{\sum}} \frac{1}{(d_j + 1)} > \overset{p}{\underset{j=1}{\sum}} \frac{1}{(d_j + 1)} \implies \\
-1 > 0,
\end{align*}
which is a contradiction.  Therefore, $\gamma$ must be negative so that the direction of the inequality switches when multiplying by $\gamma$.
We thus have:
\begin{align*}
1 > \overset{p}{\underset{j=1}{\sum}} \frac{1}{(d_j + 1)\gamma} \implies \\
\gamma < \overset{p}{\underset{j=1}{\sum}} \frac{1}{(d_j + 1)} \implies \\
-1 + \overset{p}{\underset{j=1}{\sum}} \frac{1}{(d_j + 1)} < \overset{p}{\underset{j=1}{\sum}} \frac{1}{(d_j + 1)} \implies \\
-1 < 0.
\end{align*}
\noindent Now we need to understand the conditions under which $\gamma$ could be less than zero.
The only free parameters in $\gamma$ are the values of $\sigma$:
\begin{align*}
\gamma = -1 + \overset{p}{\underset{j=1}{\sum}} \frac{1}{(d_j + 1)} &< 0 \implies \\
\overset{p}{\underset{j=1}{\sum}} \frac{1}{(d_j + 1)} &< 1 \\
\overset{p}{\underset{j=1}{\sum}} \frac{1}{\frac{\sigma_j^2 \|x\|_2^2}{f_j(x)^2}} &< 1, \quad \forall x \in \mathbb{R}^n \\
\overset{p}{\underset{j=1}{\sum}} \frac{f_j(x)^2}{\sigma_j^2 \|x\|_2^2} &< 1, \quad \forall x \in \mathbb{R}^n \\
\end{align*}
However, by the definition of the 2-induced norm of $f_j$, $\frac{|f_j(x)|}{\|x\|_2} \le \|f_j\|_{2-2}, \; \forall x \in \mathbb{R}^n$.
Therefore, using the least upper bound property of the 2-induced norm,
\begin{align}
\overset{p}{\underset{j=1}{\sum}} \frac{\|f_{q(j)}\|_{2-2}^2}{\sigma_j^2} < 1.
\end{align}
\noindent Thus, $\sigma$ may be chosen to be large enough such that $\gamma < 0$, and when $\gamma < 0$, we have shown that $\delta^2$ is always positive and thus $\delta$ is real-valued.
In addition, $\sigma$ may be easily chosen such that $\sigma_i \ge \sigma_j \ge 0$ for all $i > j$.

%We may assume this because a unitary matrix $V^* \in \mathbb{R}^{m \times m}$ could be selected such that 

Furthermore, since we have shown that $\sigma$ chosen in Equation (\ref{eq:sigma}) generates a real-valued $\delta$, $\delta$ is well-defined and, from Equation (\ref{eq:almost}), we see that that if $U = I \in \mathbb{R}^{p \times p}$ (which is unitary), then:
\begin{equation*}
f(x) = U \Sigma v(x), \; \forall x \in \mathbb{R}^n,
\end{equation*}
which completes the proof.
\end{proof}
\begin{figure}[h]
\label{fig:circles}
\centering
\includegraphics[width=\columnwidth]{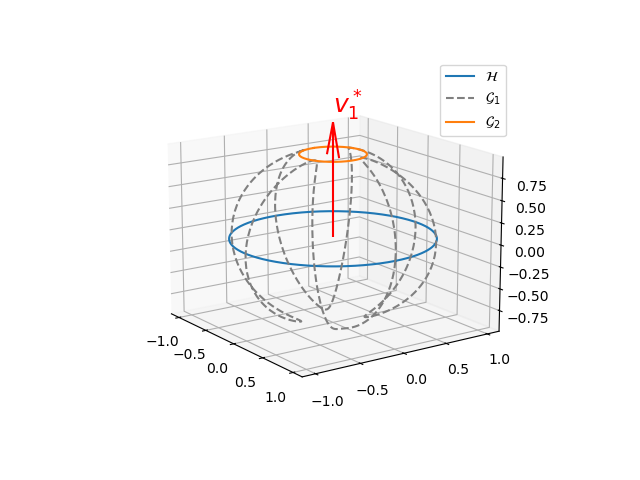}
\caption{{\bf An example of a norm-preserving mapping:}
The unit disc ${\cal H}$ in $\mathbb{R}^2$ is depicted in blue.  
Two norm-preserving mappings $g_1:\mathcal{H} \to {\cal G}_1$ (black and dashed) and $g_2: \mathcal{H} \to {\cal G}_2$ (orange and solid) are shown.
Both are liftings for distinct hypothetical functions, $f_1(x) = K \circ g_1(x)$ and $f_2(x) = K \circ g_2(x)$ (per the notation in Remark \ref{remark:V}).
Since $f_1$ and $f_2$ are functionals, $K$ can only have one non-zero singular value.
The right singular vector corresponding to this non-zero singular value is shown in red and is denoted as $v^*_1$.
Note then that $g_2$ would correspond to a function $f$ that stretches all elements of $\mathcal{H}$ uniformly.
}
\end{figure}
\begin{remark}
\label{remark:V}
In the proof, $U \in \mathbb{R}^{p \times p}$ was chosen to be identity and we had $f(x) = U \Sigma v(x)$.
However, $U$ could be selected to be an arbitrary real, unitary matrix, and an additional, real, unitary matrix $V^* \in \mathbb{R}^{m \times m}$ could be chosen as well as an injective, norm-preserving lifting, $g$, similar to the lifting given in the construction, or perhaps learned from data, such that $v(x) = V^* \circ g(x)$.
Combined with an appropriate $\Sigma$, this gives the singular value decomposition of some matrix $K \in \mathbb{R}^{m \times p}$, composed with a lifting such that:
\begin{equation}
\label{eq:kRep}
f(x) = U \Sigma (V^* \circ g)(x) = K \circ g(x),
\end{equation}
where $f$ is a bounded-input bounded-output function.
\end{remark}
\begin{remark} \label{rem:norm-preserving}
The requirement that $\|v(x)\|_2 = \|x\|_2$ is essential for identifying a meaningful $\Sigma$.
For example, if $v(x)$ were not norm-preserving (or injective), then a trivial solution would always exist where $v(x) = f(x)$, and $\Sigma = I$ such that $f(x) = I f(x)$.
Instead, the provided constraints require that a set of basis functions be identified that simultaneously 1) are collectively a norm-preserving map of the inputs into an alternate space, and 2) span the image of $f(x)$.
The existence of such a finite-dimensional mapping is not immediately obvious and is the main contribution of this theorem.
\end{remark}
\begin{remark}
Recall that the Riesz Representation Theorem \cite{rudin1987real} focuses on offering a representation of the form \mbox{$f(x) = \langle k, x \rangle$} for linear functionals of $x \in \mathbb{R}^n$, with $k$, a characterizing vector of the functional, also in $\mathbb{R}^n$.
Theorem \ref{theorem:sufficient} extends that representation to bounded-input bounded-output functionals, since Equation (\ref{eq:kRep}) can be re-written for functionals $f: \mathbb{R}^n \to \mathbb{R}$ (using the notation from Remark \ref{remark:V}) as:
\begin{equation}
f(x) = \langle k, g(x) \rangle.
\end{equation}
In the case of linear functionals, \mbox{$g: \mathbb{R}^n \to \mathbb{R}^n$} is the identity mapping, i.e. $g(x) = x$ for all $x \in \mathbb{R}^n$.
In this way, the given theorem provides a generalization of the Riesz representation to bounded-input bounded-output functionals.
The norm-preserving and injectivity properties of $g$ may render this extension useful for future work in nonlinear optimization.
\end{remark}

Note that Theorem \ref{theorem:sufficient} states that $v(x)$ can always be chosen to be injective.
This is a convenience that facilitates computing sets of inputs that are the natural relaxations of the null space and row space of linear functions.
To see this, note that the columns of $V^*$, as defined in Remark \ref{remark:V}, define a basis for the null space and row space of $K$, depending on whether their associated singular value is zero or non-zero, respectively.
If $g^{-L}$ is the left-inverse of $g$, and $\mathcal{Y}$ is the intersection of the image of $g$ with the null space of $K$, then $\{g^{-L}(y) | \; y \in \mathcal{Y}\}$ is the appropriate relaxation of the null space of $f$.

The injectivity of $v(x)$ is also a way of ensuring that all the nonlinear portions of the computations in $f$ remain reversible until the final linear computation. 
Furthermore, if $\sigma_i = 0$, then $v_i(x)$ represents information about $x$ that is lost during the computation of $f(x)$. Conversely, $v_1(x)$ represents information about $x$ that most strongly contributes to the 2-norm of $f(x)$.

Note that in the construction, the function $v$ maps from a lower dimensional space into a higher-dimensional space (sometimes referred to as a ``lifting").
This is important for maintaining injectivity without making $v(x)$ difficult to compute.
However, it is only necessary that $v$ map $\mathbb{R}^n \to \mathbb{R}^{p+1}$, as demonstrated in Lemma (\ref{lemma:nec}) in the Appendix.
Furthermore, if $p < n - 2$, then $v(x)$ need not be a ``lifting" at all, but rather a mapping more aptly called a ``lowering."
We leave the detailing of such mappings to future work, noting that they may likely be harder to compute.
In contrast, the construction given in the proof is easily computed, as will be demonstrated in Section \ref{sec:examples}.

\subsection{Bounds on Induced Norms}
There are several reasons why bounded-input bounded-output functions are a natural extension of linear functions:
\begin{itemize}
\item All linear functions are bounded-input bounded-output functions.
\item All bounded-input bounded-output functions can be bounded by a linear envelope.
\end{itemize}

\noindent This linear envelope is intricately connected to the $\sigma$ given in the construction:

\begin{corollary}
\label{corollary:upperBound}
Let $f: \mathbb{R}^n \rightarrow \mathbb{R}^p$ be an arbitrary bounded-input bounded-output function.
Then there exists a unitary matrix, $U \in \mathbb{R}^{p \times p}$, a real, non-negative, rectangular-diagonal matrix $\Sigma \in \mathbb{R}^{p \times m}$, and a norm-preserving, injective mapping, \mbox{$v: \mathbb{R}^n \rightarrow \mathbb{R}^m$}, with $m \ge p+n$, such that:
\begin{equation*}
f(x) = U \Sigma v (x), \quad \text{for all } x \in \mathbb{R}^n, \quad \forall x \in \mathbb{R}^n,
\end{equation*}
and $\sigma_1$, the maximum entry in $\Sigma$, is an upper bound on the 2-induced norm of $f$, i.e.
\begin{equation}
\|f\|_{2-2} \|x\|_2 < \|x\|_2 \sigma_1, \quad \forall x \in \mathbb{R}^n
\end{equation}
\end{corollary}
\begin{proof}
\begin{align*}
\|f\|_{2-2} &= \sup_x \frac{\|f(x)\|_2}{\|x\|_2} \\
&= \sup_x \frac{\| U \Sigma v(x) \|_2}{\|x\|_2} \\
&\le \sup_x \frac{\|U\|_{2-2} \| \Sigma \|_{2-2} \| v(x) \|_2}{\|x\|_2} \\
&= \sup_x \frac{\| \Sigma \|_{2-2} \| x \|_2}{\|x\|_2} \\
&= \| \Sigma \|_{2-2} \\
\end{align*}

By the constraint in Equation (\ref{eq:sigma}) in the construction, $\sigma_i > f_{i}$ for all $i = 1, 2, \dots p$.
This makes the inequality strict, i.e. $\|f\|_{2-2} < \sigma_1$.
\end{proof}
\noindent Thus, $\sigma$ becomes more meaningful when it is minimized during the construction of the function representation.

In linear functions, any scaling of $x^* = \sup_x \frac{\|f(x)\|_2}{\|x\|_2}$ will be stretched by the same amount, $\sigma_1$, under $f$.
In the extension to bounded-input bounded-output functions, this is no longer the case.
In bounded-input bounded-output functions, the point or set of points that achieve the induced norm of the function may be an irregular set in $\mathbb{R}^n$.
For example, in the first panel of Figure \ref{fig:complicated}, only a few inputs come close to achieving the upper bound given by $\pm \sigma_1 \|x\|_2$.
\begin{figure}[h]
\label{fig:complicated}
\centering
\includegraphics[width=\columnwidth]{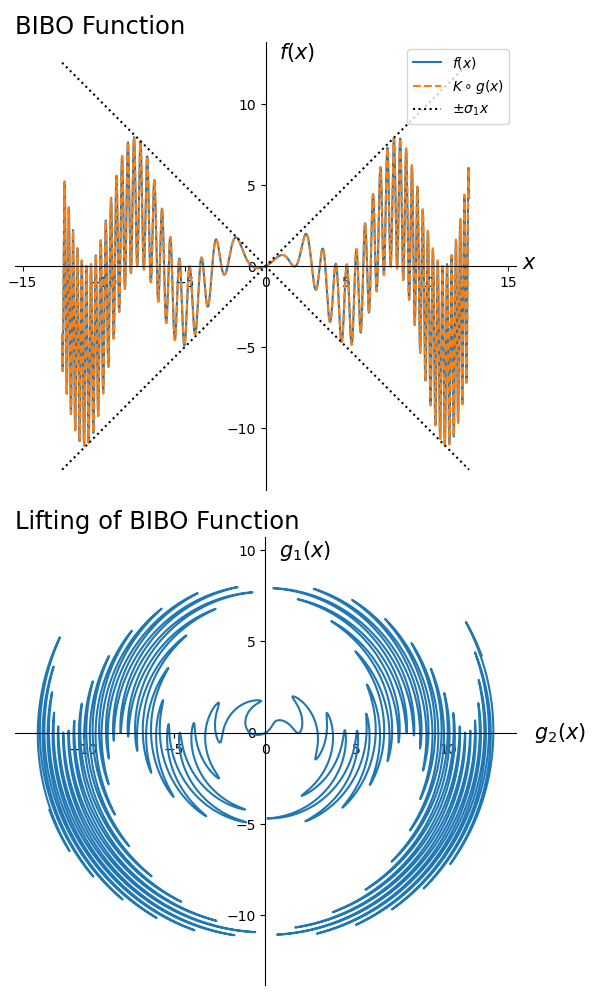}
\caption{{\bf Numerical example of the lifting proposed in this paper:}
The top panel represents the function (the blue curve) while the orange curve represents its numerical reconstruction using $f(x) = K \circ g(x)$, with $g(x)$ norm-preserving and injective.
The black dotted diagonal lines represent the upper bound on the induced norm of the function, given by $\pm \sigma_1 \|x\|_2$ (see Corollary \ref{corollary:upperBound}).
The bottom panel shows the norm-preserving, injective lifting that is given in the construction of the Theorem (\ref{theorem:sufficient}).
}
\end{figure}

\section{Examples and Experiments}
\label{sec:examples}

We tested the construction given in Theorem \ref{theorem:sufficient} on several bounded-input bounded output functions.
In each case, $f(x_i) = K g(x_i)$ for all $x_i$ tested.
The first function tested was a single-input single-output bounded-input bounded-output function with a 2-induced norm of 1:
\begin{equation}
f(x) = \frac{x \sin(x) + x \cos(x^2)}{2}.
\end{equation}
The results of this experiment, as well as a visualization of the computed lifting, can be seen in Figure \ref{fig:complicated}.

We also visualized the lifting for a multi-input single-output function, which, for convenience, we will write as several component functions, listed below:
\begin{align*}
h_1(x) &= \sin(0.1 x_1 * x_2), \\
h_2(x) &= 0.1 \cos(3 \frac{x_1/x_2}), \\
h_3(x) &= 0.4 \sin(20 x_1), \\
h_4(x) &= 0.3 \cos(x_2 + 4), \\
h_5(x) &= 0.3 \sin(0.1 e^{x_1}), \\
h_6(x) &= 0.2 \cos(\frac{1}{x_1^2}), \\
h_7(x) &= 0.1 \sin(0.1 (x_1 + x_2)), \\
h_8(x) &= 0.1 \cos(0.001 x_2^2),
\end{align*}
such that:
\begin{equation}
f(x) = \frac{\|x\|_2}{2.5} \overset{8}{\underset{i=1}{\sum }}h_i(x).
\end{equation}
The function and computed lifting can be visualized in Figure \ref{fig:multivariate}.

\begin{figure}[h]
\label{fig:multivariate}
\centering
\includegraphics[width=\columnwidth]{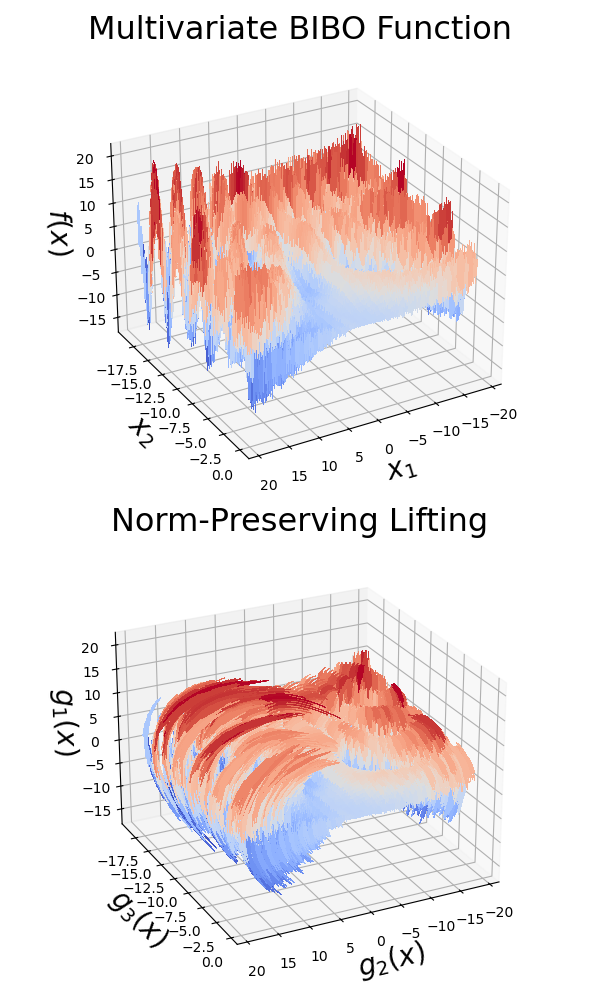}
\caption{{\bf Numerical example of the lifting proposed in this paper:}
The top panel represents a multi-input single-output bounded-input bounded-output function.
The bottom panel shows the norm-preserving lifting computed using the construction given in this paper.
Let $K = U \Sigma V^*$ be the SVD of $K$.
Then the functional $g_3(x)$ is the projection of the lifting $g$ onto $v_1^*$, which corresponds to $\sigma_1 = 1$ in this instance, and thus lifting further in $g_3(x)$ corresponds to increased stretching of $x$.
In contrast, $g_2(x)$ and $g_3(x)$ correspond with the projections of the lifting onto $v_2^*$ and $v_3^*$, which correspond to $\sigma_2 = \sigma_3 = 0$, i.e. $g_1$ and $g_2$ store information that is lost from $x \rightarrow f(x)$.
}
\end{figure}

\section{Conclusion}

Here we have demonstrated that every bounded-input bounded-output function, $f:\mathbb{R}^n \to \mathbb{R}^p$, has an SVD-like decomposition with a finite-dimensional representation of the form:
\begin{equation*}
f(x) = U \Sigma v(x)
\end{equation*}
By leveraging an injective ``lifting," i.e. $v: \mathbb{R}^n \to \mathbb{R}^{n+p}$, this decomposition facilitates the computation of the extensions of null and row spaces for bounded-input and bounded-output functions.
A constraint on the lifting to be norm-preserving causes the 2-induced norm of $f$ to be upper-bounded by the maximum element of $\Sigma$.
When $p = 1$, the representation also provides a natural extension of the Reisz Reprsentation Theorem to bounded-input bounded-output functionals.

\section{Acknowledgments}
We express our gratitude to Tyler Jarvis, Mark Transtrum, and Jared Whitehead for their invaluable internal reviews of the early stages of this work.
%Funding for this work was provided by Achilles Heel Technologies in part, and by the Army Research Office.
% Enoch added the Army Research Office part and I'm not sure if that should stay or go since we put funding info by the authors... 
\bibliographystyle{plain}
\bibliography{ifacconf}

\section{Appendix}
\begin{lemma}
\label{lemma:nec}
There does not exist an injective, norm-preserving mapping, \mbox{$g: \mathbb{R}^n \rightarrow \mathbb{R}^p$}, with an associated matrix, $K \in \mathbb{R}^{p \times p}$, satisfying
\begin{equation}\label{eq:lemma1eq}
f(x) = K \circ g(x)
\end{equation}
for all bounded-input bounded-output functions, $f: \mathbb{R}^n \to \mathbb{R}^p$, with $n, p \in \mathbb{N}^+$.
\end{lemma}
\begin{proof}
Consider a bounded-input bounded-output function, $f: \mathbb{R}^n \rightarrow \mathbb{R}^p$, such that
\begin{align}
f(x) = 
\begin{bmatrix}
f_1(x) \\
f_2(x) \\
\vdots \\
f_p(x)
\end{bmatrix}
= 
\begin{bmatrix}
a_1 \|x\|_2 \\
a_2 \|x\|_2 \\
\vdots \\
a_p \|x\|_2
\end{bmatrix}
\end{align}

Note that $f(x)$ is trivially bounded-input bounded-output since $\|f_{q(i)}(x)\|_2 \le a_i \|x\|_2$ for $i = 1, 2, \dots, n$.

Let $g: \mathbb{R}^n \rightarrow \mathbb{R}^p$ be any norm-preserving, injective mapping, i.e. $\|g(x)\|_2 = \|x\|_2$ for all $x \in \mathbb{R}^n$, and, for all $x_1 \neq x_2 \in \mathbb{R}^n, g(x_1) \neq g(x_2)$.
One consequence of this last property is that for a set with no repeats, $\mathcal{X} \subset \mathbb{R}^n$, with a given cardinality, $c$, the set $\{g(x) | x \in \mathcal{X}\}$  must have the same cardinality when repeats are removed.
Let $|\cdot|$ be an operation measuring the repeat-removed cardinality of a set.
Then, $|\mathcal{X}| = |\{g(x) | x \in \mathcal{X}\}| = c$.

Because $g(x)$ is a norm-preserving mapping, we can, without loss of generality, consider the actions of $f$ and $g$ on a set of inputs, $\mathcal{H}_r \subset \mathbb{R}^n$, such that $x \in \mathcal{H}_r \implies \|x\|_2 = r$.
Now consider the set $\mathcal{G}_r = \{g(x)| x \in \mathcal{H}_r\}$.
Note that if $n > 1$, then $|\mathcal{G}_r| = |\mathcal{H}_r| = 2^{\aleph_0}$, and $\mathcal{G}_r$ is necessarily a subset of an origin-centered hypersphere of radius $r$ and dimension $p$. 

Define $K \in \mathbb{R}^{p \times p}$, an arbitrary linear function, such that its singular value decomposition is given by $K = U \Sigma V^*$, with the standard definitions of $U, \Sigma$, and $V^*$.
Now suppose that $f(x) = (K \circ g)(x)$.
Without loss of generality, we may consider $U = I$ and $V^* = I$, with $I$ of the appropriate dimensions.
With these values of $U$ and $V^*$, $f(x) = (K \circ g)(x) \implies f_{q(i)}(x) = \sigma_i g_i(x)$.

We now note a relationship between the image of $g$ and the pre-image of $K$ that must be satisfied for $f(x) = (K \circ g)(x)$ to hold.
Consider $\mathcal{Y}$, the set of vectors such that $Ky = f(x)$ for $y \in \mathcal{Y}$ and $x \in \mathcal{H}_r$.
Given our choice of $f$, $f(x)$ is identical for all $x \in \mathcal{H}_r$.
If $n > 1$ (i.e. $|\mathcal{H}_r| = 2^{\aleph_0}$), then to simultaneously satisfy injectivity \emph{and} representation, the intersection of $\mathcal{Y}$ and $\mathcal{G}_r$ must have the cardinality of the continuum, i.e. $|\mathcal{G}_r \cap \mathcal{Y}| = 2^{\aleph_0}$.

Now let $P_i(g(x))$ be the projection of $g(x)$ onto the one-dimensional subspace spanned by $v^*_i$.
Given the norm-preserving constraints on $g(x)$, we then have that $P_i(g(x)) = b_i \|x\|_2$, for some $-1 \le b_k \le 1$.
Note that since there is no upper bound on $\sigma$, $\sigma_i$ and $g_i$ (and therefore $b_i$) can always be chosen such that $\sigma_i b_i = a_k$, implying that the intersection of the image of $g(x)$ for $x \in \mathcal{H}_r$ and $\mathcal{Y}$ can be non-empty.

Let $\mathcal{Y}_i \in \mathcal{H}$ be the set of all vectors reachable from $g$ and in the pre-image of $K$ satisfying

\begin{align*}
Ky_i =
\begin{bmatrix}
c_1 \\
c_2 \\
\vdots \\
\sigma_i g_i(g^{-L}(y_i))  \\
\vdots  \\
c_p
\end{bmatrix},
\quad \text{for } y_i \in \mathcal{Y}_i
\end{align*}

with $c_{j \neq i} \in \mathbb{R}^p$ not specified.

Thus, for $y \in \mathcal{Y} = (\overset{p}{\underset{i}{\cap}} \mathcal{Y}_i)$,
\begin{equation}
Ky = 
\begin{bmatrix}
\sigma_1 g_1(g^{-L}(y))  \\
\sigma_2 g_2(g^{-L}(y))  \\
\vdots \\
\sigma_p g_p(g^{-L}(y))  \\
\end{bmatrix}
\end{equation}

and furthermore $\mathcal{H}_r \cap \mathcal{Y} = \emptyset$, unless $\sigma$ is chosen such that $\sigma_i P_i(g_i(g^{-L}(y))) = a_i\|g_i^{-L}(y)\|_2$, in which case $|\mathcal{H}_r \cap \mathcal{Y}| = 1$, see Figure \ref{fig:lemma}.
This non-infinite cardinality violates the injectivity of $g$ since $|\mathcal{G}| = 2^{\aleph_0}$, implying that an injective, norm-preserving mapping $g(x)$ can only satisfy $f(x) = (K \circ g)(x)$ for a single $x$ when $f_{q(i)}(x) = a_i \|x\|_2$.
\end{proof}

\begin{figure}[h]
\label{fig:lemma}
\centering
\includegraphics[width=0.7\columnwidth]{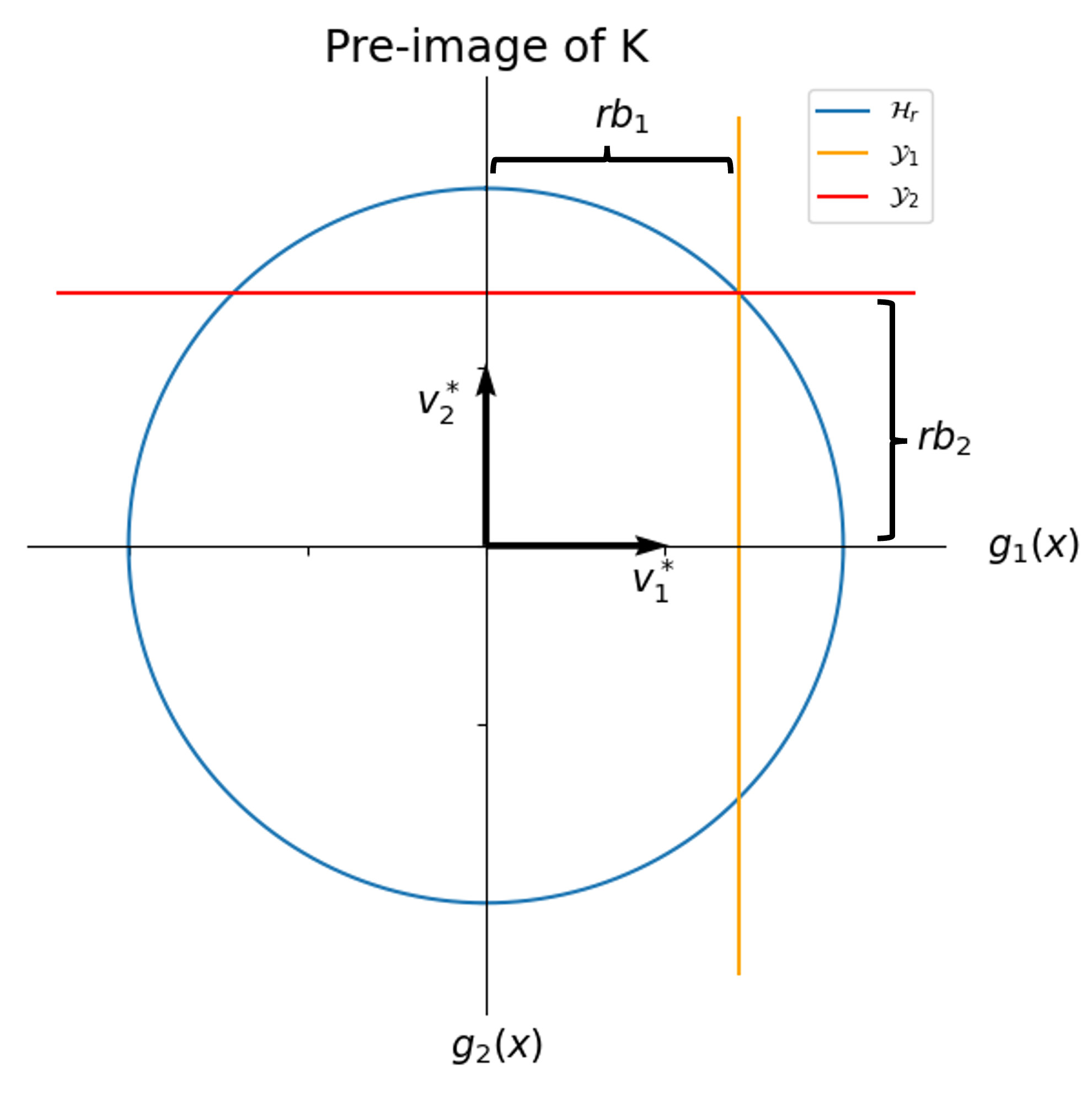}
\caption{{\bf 2-dimensional demonstration of the geometric reasoning in Lemma \ref{lemma:nec}:} In this example, $g:\mathbb{R}^2 \to \mathbb{R}^2$, with $f:\mathbb{R}^2 \to \mathbb{R}^2$.
The blue, yellow, and red curves correspond to the sets $\mathcal{H}_r$, $\mathcal{Y}_1$, and $\mathcal{Y}_2$ mentioned in the proof.
The key points of the proof are 1) to be norm-preserving, $g$ must map $\mathcal{H}_r \to \mathcal{H}_r$, 2) to satisfy representation of $f(x) = a \|x\|_2$ (an arbitrary BIBO function) $g$ must map $\mathcal{H}_r \to \mathcal{Y}_1$ and simultaneously $\mathcal{H}_r \to \mathcal{Y}_2$, and 3) the cardinality of the intersection of $\mathcal{H}_r$, $\mathcal{Y}_1$, and $\mathcal{Y}_2$ is at most 1.
This contradicts the injectivity of $g$, since, in order for $g$ to be injective, the set $\{g(x)|x \in \mathcal{H}_r\}$ must have cardinality of the continuum, $2^{\aleph_0}$.
Therefore a norm-preserving, injective mapping from $\mathbb{R}^n \to \mathbb{R}^n$ composed with $K \in \mathbb{R}^{n \times n}$ cannot represent all bounded-input bounded-output functions.
}
\end{figure}

\end{document}